\documentclass{amsart}
\usepackage[utf8]{inputenc}

\title[Stochastic PDEs associated with the Forward Equations]{Existence and Uniqueness of Stochastic PDEs associated with the Forward Equations: An Approach using Alternate Norms}

\author{Suprio Bhar$^\ast$}
\address{Suprio Bhar, Department of Mathematics and Statistics, Indian Institute of Technology Kanpur, Kalyanpur, Kanpur - 208016, India.}
\email{suprio@iitk.ac.in}
\thanks{$^\ast$Corresponding Author}

\author{Rajeev Bhaskaran}
\address{Rajeev Bhaskaran, School of Mathematics, Indian Institute of Science Education and Research Thiruvananthapuram Campus, Maruthamala P. O, Vithura, Kerala - 695551, India.}
\email{brajeev@iisertvm.ac.in}

\author{Arvind Kumar Nath}
\address{Arvind Kumar Nath, Department of Mathematics and Statistics, Indian Institute of Technology Kanpur, Kalyanpur, Kanpur - 208016, India.}
\email{yarvind@iitk.ac.in}

\usepackage{setspace}
\usepackage{geometry}
\usepackage{enumerate}
\usepackage{enumitem, xcolor, amssymb,latexsym,amsmath,bbm}
\usepackage{mathtools}
\usepackage[mathscr]{euscript}
\usepackage{amsmath}
\usepackage{amssymb}
\usepackage{enumitem} 
\usepackage{cite}
\newcommand{\R}{\mathbb{R}}

\newcommand{\Exp}{\mathbb{E}}
\newcommand{\F}{\mathcal{F}}
\newcommand{\Sc}{\mathcal{S}}

\newcommand{\Z}{\mathbb{Z}}
\newcommand{\E}{\mathcal{E}}

\newcommand{\pp}{\mathbb{P}}
\newcommand{\inpr}[3][]{\left\langle#2 \,,\, #3\right\rangle_{#1}}

\newcommand{\oldinpr}[3][]{{\lll\!#2 \,,\, #3\!\ggg_{#1}}}

\newcommand{\oldnorm}[2][]{{|\!\|#2\|\!|_{#1}}}

\newtheorem{theorem}{Theorem}[section]
\newtheorem{corollary}[theorem]{Corollary}
\newtheorem{lemma}[theorem]{Lemma}
\newtheorem{proposition}[theorem]{Proposition}
\newtheorem{definition}[theorem]{Definition}
\newtheorem{remark}[theorem]{Remark}

\numberwithin{equation}{section}

\allowdisplaybreaks

\begin{document}

\begin{abstract}
We consider stochastic PDEs
\[dY_t = L(Y_t)\, dt + A(Y_t).\, dB_t, t > 0\]
and associated PDEs
\[du_t = L u_t\, dt, t > 0\]
with regular initial conditions. Here, $L$ and $A$ are certain partial differential operators involving multiplication by smooth functions and are of the order two and one respectively, and in special cases are associated with finite dimensional diffusion processes. This PDE also includes Kolmogorov's Forward Equation (Fokker-Planck Equation) as a special case. We first prove a Monotonicity inequality for the pair $(L, A)$ and using this inequality, we obtain the existence and uniqueness of strong solutions to the Stochastic PDE and the PDE. In addition, a stochastic representation for the solution to the PDE is also established.
\end{abstract}
\keywords{Hermite-Sobolev space, Strong solution, Stochastic Differential Equation, Stochastic Partial Differential Equation, Evolution Equation, Monotonicity Inequality, Stochastic representation, Second order linear differential operators}
\subjclass[2010]{60H10, 60H15, 60H30}

\maketitle

\section{Introduction}\label{S1:intro}

Linear Parabolic equations have been closely associated with the theory of Stochastic Differential Equations (SDEs) from its very inception - mainly through the `forward' and `backward'
Kolmogorov equations - so that this relationship could by now be considered `classical'. We refer to \cite{MR2190038, MR2932517, MR2295424, MR1329542, MR1121940, MR3236753, MR771478, MR1472487, MR570795, MR876085, MR2373102, MR1985790, MR2111320, MR2857016, MR2765423, MR2674056} for more details on this topic.
More recently, non-linear Fokker-Planck equations and Stochastic Fokker-Planck Equations have also been studied (for example, see \cite{MR4583740, MR4600818}). In \cite{MR2373102} it was shown that the forward equation can be obtained from the linear Stochastic Partial Differential Equation (SPDE) satisfied by the flow of an SDE by simply taking expectations in an appropriate Hilbert space (the Hermite Sobolev spaces
$(\Sc_p(\R^d), \oldinpr[p]{\cdot}{\cdot}))$ in which the SPDE lives. These Hilbert spaces are contained in $\Sc^\prime(\R^d)$, the space of tempered distributions, which is the dual space of $\Sc(\R^d)$, the space of $\R$-valued rapidly decreasing smooth functions on $\R^d$ (referred hereafter as the Schwartz space) with the topology given by L. Schwartz (see \cite{MR1681462, MR2296978, MR771478, MR1157815} and the references therein).

The SPDEs considered in \cite{MR2373102} are of the form
\begin{equation}\label{main-adjoint-spde}
dY_t = L^\ast(Y_t)\, dt + A^\ast(Y_t).\, dB_t, t > 0
\end{equation}
where
\begin{enumerate}[label=(\roman*)]
    \item $Y_0$ is $\F_0$ measurable random variable taking values in the space of compactly supported distributions $\E^\prime(\mathbb{R}^d)$,
    \item $\{B_t\}_t$ is an $(\F_t)_t$ adapted $d$-dimensional standard Brownian motion with the components given by $B_t = (B^{(1)}_t, \cdots, B^{(d)}_t)$,
    \item $A^\ast = (A_1^\ast, A_2^\ast, \cdots, A_d^\ast)$ with $A_1^\ast,\cdots, A_d^\ast:\Sc^\prime(\R^d) \to \Sc^\prime(\R^d)$ are linear `adjoint' differential operators of order one and $L^\ast:\Sc^\prime(\R^d) \to \Sc^\prime(\R^d)$ is a linear `adjoint' differential operator of order two (see \eqref{L-A-adjoint}).
\end{enumerate}

The question of uniqueness of solutions of the SPDE and the solutions of the PDE obtained as above however was left open in \cite{MR2373102}. In \cite{MR4117986} the uniqueness was established when the initial conditions were smooth by proving the `monotonicity inequality' (see \cite{MR570795}) in $\mathcal{L}^2(\R^d)$ for the pair of operators $(L^\ast, A^\ast)$ given as in \cite{MR2373102}. The use of this inequality to prove existence and uniqueness of SPDEs was earlier studied in \cite{MR2479730, brajeev-arxiv, MR3063763, MR2560625, MR2590157, MR3331916, MR4568882, MR1465436}.

In this paper we adopt a new approach to prove the Monotonicity inequality. Since the operators $(L^\ast, A^\ast)$ involve multiplication by functions and since these are somewhat difficult to handle in the spaces $(\Sc_p, \oldinpr[p]{\cdot}{\cdot}))$, we use a different but `equivalent' inner product $\inpr[p]{\cdot}{\cdot}$ on $\Sc_p(\R^d)$ (the corresponding norms are equivalent, see Proposition \ref{norm-equivalence}). We prove the Monotonicity inequality in this set up and prove a stronger version of uniqueness, namely, uniqueness in $\Sc_p(\R^d), p > 0$. Our approach looks at a more general linear SPDE with coefficients $(L, A)$ as in \eqref{L-A} and of the form
\begin{equation}\label{main-spde}
dY_t = L(Y_t)\, dt + A(Y_t).\, dB_t, t > 0
\end{equation}
where $Y_0$ is an $\Sc_p(\R^d)$ valued $\F_0$ measurable random variable and derives the SPDE associated with the forward equation with $(L^*,A^*)$ as a special case.

The organization of this paper is as follows. We first recall necessary facts about the topology on $\Sc(\R^d)$ in Subsection \ref{S2:topology} and discuss the main results in the remaining part of Section \ref{S2:description}. Proofs of the results have been discussed in Section \ref{S3:Proofs}.

\section{Notations and Main results}\label{S2:description}

\subsection{Countably Hilbertian Nuclear topology on the Schwartz space}\label{S2:topology}
Let $\mathbb{Z}^d_+:=\{\alpha=(\alpha_1,\cdots, \alpha_d): \; \alpha_i \text{ non-negative integers}\}$. If $\alpha\in\mathbb{Z}^d_+$, we define $|\alpha|:=\alpha_1+\cdots+\alpha_d$. The topology on $\Sc(\R^d)$ is given by a family of seminorms $p_n, n = 0, 1, \cdots$ (see \cite{MR771478, MR1681462}) where
\[p_n(f) := \sup_{x \in \R^d} \left[(1 + |x|)^n \max_{\alpha : |\alpha| \leq n} \left|\frac{\partial^{|\alpha|}}{\partial_{x_1}^{\alpha_1}\cdots \partial_{x_d}^{\alpha_d}}f(x) \right|\right], \forall f \in \Sc(\R^d)\]
where $|x|$ denotes the usual Euclidean norm for $x \in \R^d$. We shall adopt the short-hand notation $\partial^\alpha$ and $\partial_j$ for $\frac{\partial^{|\alpha|}}{\partial_{x_1}^{\alpha_1}\cdots \partial_{x_d}^{\alpha_d}}$ and $\frac{\partial}{\partial_{x_j}}$, respectively.

For $p \in \R$, consider the increasing family of norms $\oldnorm[p]{\cdot}$, defined by the inner products
\begin{equation}
\oldinpr[p]{f}{g}
:=\sum_{n\in\mathbb{Z}^d_+}(2|n|+d)^{2p}\langle f,h_n\rangle_0 \langle g,h_n\rangle_0,\ \ \ f,g\in\Sc(\R^d).
\end{equation}
In the above equation, $\{h_n: n\in\mathbb{Z}^d_+\}$ is an orthonormal basis for $\mathcal{L}^2(\R^d)$ given by the Hermite functions and $\langle\cdot,\cdot\rangle_0$ is the usual
inner product in $\mathcal{L}^2(\R^d)$. The Hermite-Sobolev spaces $\Sc_p(\R^d), p \in \R$ are defined as the completion of $\Sc(\R^d)$ in
$\oldnorm[p]{\cdot}$. Note that the dual space $\Sc_p^\prime(\R^d)$ is isometrically isomorphic with $\Sc_{-p}(\R^d)$ for $p\geq 0$. For $\phi \in \Sc_p(\R^d)$ and $\psi \in \Sc^\prime(\R^d)$, we write the duality action by $\inpr{\psi}{\phi}$. We also have $\Sc(\R^d) = \bigcap_{p}(\Sc_p(\R^d),\oldnorm[p]{\cdot}), \Sc^\prime(\R^d) = \bigcup_{p>0}(\Sc_{-p}(\R^d),\oldnorm[-p]{\cdot})$ and $\Sc_0(\R^d) = \mathcal{L}^2(\R^d)$. The following basic relations hold for the $\Sc_p(\R^d)$ spaces: for $0<q<p$, \[\Sc(\R^d) \subset \Sc_p(\R^d) \subset \Sc_q(\R^d) \subset \mathcal L^2(\R^d) = \Sc_0(\R^d) \subset \Sc_{-q}(\R^d)\subset \Sc_{-p}(\R^d)\subset \Sc^\prime(\R^d).\]
The topology on $\Sc(\R^d)$ given by the norms $\oldnorm[p]{\cdot}, p \in \Z_+$ is the same as the usual topology on $\Sc(\R^d)$ (see \cite[Proposition 1.1]{MR1837298}).

We now recall another way to describe the same topology on $\Sc(\R^d)$. For $p \in \Z_+$, define for $\phi, \psi \in \Sc(\R^d)$,
\[\inpr[p]{\phi}{\psi}:= \sum_{|\alpha| + |\beta| \leq 2p}\int_{\R} x^\alpha\partial^\beta \phi(x)\, x^\alpha\partial^\beta \psi(x)\, dx,\]
where $\alpha,\beta\in\Z^d_+$ in the above sum. We denote the corresponding norms by $\|\cdot\|_p$. It is known that the topology on $\Sc(\R^d)$ generated by $\|\cdot\|_p, p \in \Z_+$ is the same as the usual topology. We recall the next result from \cite[Proposition 3.3]{MR1999259} (also see \cite[Remark 1.3.1]{MR771478}).

\begin{proposition}\label{norm-equivalence}
For all $p \in \Z_+$, there exist constants $C_1 = C_1(d, p) > 0$ and $C_2 = C_2(d, p) > 0$ such that
\[\oldnorm[p]{\phi} \leq C_1 \|\phi\|_p \leq C_2 \oldnorm[p]{\phi}, \forall \phi \in \Sc(\R^d).\]
\end{proposition}

Consequently, for $p \in \Z_+$, completing $\Sc(\R^d)$ with the inner-product $\inpr[p]{\cdot}{\cdot}$ gives us the same Hermite-Sobolev spaces $\Sc_p(\R^d)$. Unless stated otherwise, for $p \in \Z_+$ we shall work with the norm $\|\cdot\|_p$.

Consider the derivative map denoted by $\partial_j:\Sc(\R^d) \to
\Sc(\R^d), j = 1, 2, \cdots, d$. We can extend this map by duality to
$\partial_j:\Sc'(\R^d) \to \Sc'(\R^d), j = 1, 2, \cdots, d$ as follows: for $\psi_1 \in
\Sc'(\R^d)$,
\[\inpr{\partial \psi_1}{\psi_2} := -\inpr{\psi_1}{\partial \psi_2}, \; \forall \psi_2
\in \Sc(\R^d).\]
It is well-known that $\partial_j:\Sc_p(\R^d) \to \Sc_{p-1}(\R^d)$ and $\partial^2_{ij}:\Sc_p(\R^d) \to \Sc_{p-1}(\R^d)$ are bounded linear operators, for $i, j = 1, 2, \cdots, d$.

In our arguments, we use the following fact. Let $f: \R^d \to \R$ be a smooth function such that its derivatives are bounded. In particular, $f$ has linear growth. In this case, we have the multiplication map $M_f:\Sc(\R^d) \to \Sc(\R^d)$ defined by $M_f(\phi) := f\phi, \forall \phi \in \Sc(\R^d)$. This map extends by duality to $M_f:\Sc^\prime(\R^d) \to \Sc^\prime(\R^d)$ as follows: 
\[\inpr{M_f \psi_1}{\psi_2} := \inpr{\psi_1}{M_f \psi_2}, \; \forall \psi_1 \in \Sc^\prime(\R^d), \psi_2
\in \Sc(\R^d).\]

We also use the following multiplication by `monomials' of the form $x^\alpha := x_1^{\alpha_1} x_2^{\alpha_2} \cdots x_d^{\alpha_d}$ given by $x^\alpha :\Sc(\R^d) \to \Sc(\R^d)$ defined by $(x^\alpha \phi)(y) := y^\alpha \phi(y), \forall \phi \in \Sc(\R^d), y \in \R^d$. This map is extended by duality to $x^\alpha: \Sc^\prime(\R^d) \to \Sc^\prime(\R^d)$ as follows: 
\[\inpr{x^\alpha \psi_1}{\psi_2} := \inpr{\psi_1}{x^\alpha \psi_2}, \; \forall \psi_1 \in \Sc^\prime(\R^d), \psi_2
\in \Sc(\R^d).\]

\subsection{Main results}\label{S2:results}

Consider the following operators on $\Sc^\prime(\R^d)$. For $\phi \in \Sc^\prime(\R^d)$, we define 
\begin{equation}\label{L-A}
\begin{split}
A \phi &:= (A_1 \phi, A_2\phi, \dots, A_d\phi),\\
A_i\phi &:= \sum_{j=1}^d\sigma_{ji} \partial_j \phi + h_i \phi, i = 1, 2, \cdots, d,\\
L\phi &:= \frac{1}{2}\sum_{i,j=1}^d ( \sigma \sigma^t )_{ij} \partial^2_{ij} \phi + \sum_{j=1}^d f_j \partial_j \phi + g \phi, 
\end{split}
\end{equation}
where $g, h_j, \sigma_{ij}: \R^d \to \R, i, j  = 1, \cdots, d$ are bounded functions with bounded derivatives, and $f_j: \R^d \to \R, j  = 1, \cdots, d$ have bounded derivatives. We write $\sigma := (\sigma_{ij})_{d \times d}$ with $\sigma^t$ denoting the transpose of the matrix $\sigma$ of functions. We shall also refer to the functions $f, h: \R^d \to \R^d$ with component functions given by $f_j$'s and $h_j$'s, respectively.

Note that $L\phi \in \Sc(\R^d)$ and $A_i \phi \in \Sc(\R^d), i = 1, \cdots, d$ whenever $\phi \in \Sc(\R^d)$. The next result describes the boundedness properties for these operators in $\|\cdot\|_p$. This is an analogue of \cite[Proposition 3.2]{MR2373102}, where $L$ and $A$ were taken in an adjoint form.

\begin{proposition}\label{L-A-bound}
Let $g, h_j, f_j$ and $\sigma_{ij}$ be as above. For any positive integer $p$, the linear operators $L:\Sc_p \to \Sc_{p - 1}$ and $A_i: \Sc_p \to \Sc_{p - 1}, i = 1, 2, \cdots, d$ are bounded.
\end{proposition}

\begin{proof}
Fix multi-indices $\alpha, \beta$ such that $|\alpha| + |\beta| \leq 2(p - 1)$. Now, for any $i, j  = 1, \cdots, d$,
\[\left\|x^\alpha \partial^\beta\left(( \sigma \sigma^t )_{ij} \partial^2_{ij} \phi\right)\right\|_0^2 \leq C(\sigma, d) \sum_{|\alpha| + |\gamma| \leq 2p} \left\|x^\alpha \partial^\gamma \phi\right\|_0^2 = C(\sigma, d) \|\phi\|_p^2,\]
for all $\phi \in \Sc_{p - 1}$. Here, we use the fact that $\sigma_{ij}$ and its derivatives are bounded and $C(\sigma, d)$ denotes a positive constant depending on $\sigma$ and $d$. Again, $f_j$ has bounded derivatives and in particular $|f_j(x)| \leq C(f) (1 + |x|), \forall x \in \R^d$, where $C(f)$ is some positive constant depending on $f$. Consequently, 
\[\left\|x^\alpha \partial^\beta\left(f_j \partial_{j} \phi\right)\right\|_0^2 \leq C(f, d) \sum_{|\alpha| + |\gamma| \leq 2p} \left\|x^\alpha \partial^\gamma \phi\right\|_0^2 = C(f, d) \|\phi\|_p^2,\]
for all $\phi \in \Sc_{p - 1}$. Here, $C(f, d)$ denotes a positive constant depending on $f$ and $d$.

Arguing in a similar manner, the necessary bounds on the remaining terms in $L$ and $A_i$ follow.
\end{proof}

\begin{definition}[Monotonicity inequality]\label{Mon-ineq}
Fix $p \in \Z_+$. We say that the pair of operators $(L, A)$ satisfy the Monotonicity inequality in $\|\cdot\|_p$ if we have
\begin{equation}\label{Monotoniticity-inequality}
2\inpr[p]{\phi}{L\phi} + \|A\phi\|_{HS(p)}^2\leq C\|\phi\|^2_p, \forall \phi\in \Sc(\R^d)
\end{equation}
where $C = C(d, p, \sigma, f, g, h) > 0$ is a positive constant not depending on $\phi$ and $\|A\phi\|_{HS(p)}^2 := \sum_{i = 1}^d \|A_i\phi\|_{p}^2$.
\end{definition}

To the best of our knowledge, explicit verification of this inequality has been attempted only with $\inpr[p]{\cdot}{\cdot}$ and $\|\cdot\|_p$ replaced by $\oldinpr[p]{\cdot}{\cdot}$ and $\oldnorm[p]{\cdot}$, respectively. This inequality for the constant coefficient case was first proved in \cite[Theorem 2.1]{MR2590157}. Using an identification of the adjoint of the derivative operator on $\Sc^\prime(\R^d)$, a more conceptual proof was given in \cite[Theorems 2.1 and 3.1]{MR3331916}. This identification lead to the proof of the inequality when $(L, A)$ are in the following adjoint form: for $\phi \in \Sc^\prime$, we define
\begin{equation}\label{L-A-adjoint}
\begin{split}
A^\ast \phi &:= (A_1^\ast \phi, A_2^\ast\phi, \dots, A_d^\ast\phi),\\
A_i^\ast\phi &:= -\sum_{j=1}^d \partial_j (\sigma_{ji}\phi), i = 1, 2, \cdots, d,\\
L^\ast\phi &:= \frac{1}{2}\sum_{i,j=1}^d  \partial^2_{ij} \left(( \sigma \sigma^t )_{ij}\phi\right) - \sum_{j=1}^d \partial_j (b_j\phi),
\end{split}
\end{equation}
with $\sigma = (\sigma_{ij})_{d \times d}$ being a constant matrix and $b = (b_1, \cdots, b_d)$ with each $b_i$ being an affine function. We use `$\ast$' especially to highlight the adjoint form. In \cite[Theorem 4.4]{MR4117986}, this inequality was proved for $(L^\ast, A^\ast)$ when $p = 0$, $\sigma_{ij}$'s are bounded smooth functions with bounded derivatives and $b_i$'s are smooth functions with bounded derivatives.

In this article, we work with the Monotonicity inequality involving $\inpr[p]{\cdot}{\cdot}$ and $\|\cdot\|_p$, as stated in Definition \ref{Mon-ineq}. The following is the main result of this article.

\begin{theorem}\label{variable-coeff-monotonicity}
Consider the operators $L$ and $A$ as in \eqref{L-A}, where $g, h_j, \sigma_{ij}: \R^d \to \R, i, j  = 1, \cdots, d$ are bounded functions with bounded derivatives, and $f_j: \R^d \to \R, j  = 1, \cdots, d$ have bounded derivatives. In particular, $f_j$'s satisfy the linear growth condition. Then the pair $(L, A)$ satisfies the Monotonicity inequality for any $p \in \Z_+$.
\end{theorem}

The method of proof of Theorem \ref{variable-coeff-monotonicity} has been discussed in Subsection \ref{S2:method} below. We also provide some applications of this result to Kolmogorov's forward equation in Subsection \ref{S2:Applications-SPDE}, where the operators $L$ and $A$ are in the adjoint form mentioned in \eqref{L-A-adjoint}. Consequently, we also discuss about probabilistic representations of associated PDEs in Subsection \ref{S2:Applications-PDE}.

\subsection{Methodology}\label{S2:method}

The proof of Theorem \ref{variable-coeff-monotonicity} hinges on the following crucial lemma.

\begin{lemma}\label{order-reduction}
Let $f: \R^d \to \R$ be a smooth function with bounded derivatives. Let $k = (k_1,\cdots, k_d)$ be a multi-index such that for some $j = 1, \cdots, d$, we have $k_j \geq 1$ and $|k - e_j|$ is odd. Here, $e_j$ denotes the standard basis vector in $\R^d$ with $j$-th entry equal to $1$ and all other entries being $0$. Then for any $\phi \in \Sc(\R^d)$, we have
\[ \inpr[0]{f\partial^\beta\phi}{\partial^{\beta+(k-e_j)}\phi} =-\frac{1}{2}\sum_{  0 < r \leq k-e_j}\binom{k-e_j}{r}\inpr[0]{\partial^r f\partial^{\beta+(k-e_j)-r}\phi}{\partial^\beta\phi}.\]
Here, $\beta \leq \alpha$ for multi-indices mean that $\beta_j \leq \alpha_j, \forall j = 1, \cdots, d$ and $\binom{\alpha}{\beta} = \prod_{j = 1}^d \binom{\alpha_j}{\beta_j}$.

In particular, for dimension one the result reads as follows. Let $f:\R \to \R$ be a smooth function with bounded derivatives. Then for any positive integer $k$, we have
 \[ \inpr[0]{f\partial^\beta\phi}{\partial^{\beta+(2k-1)}\phi} =-\frac{1}{2}\sum_{r=1}^{2k-1}\binom{2k-1}{r}\inpr[0]{\partial^r f\partial^{\beta+(2k-1)-r}\phi}{\partial^\beta\phi}, \forall \phi\in \Sc(\R).\]
\end{lemma}

\begin{remark}
In dimension one, for $k=1$, Lemma \ref{order-reduction} implies 
 \[ \inpr[0]{f\partial^\beta\phi}{\partial^{\beta+1}\phi} =-\frac{1}{2}\inpr[0]{\partial f\partial^\beta\phi}{\partial^\beta\phi}, \forall \phi\in \Sc(\R),\]
for any smooth function $f: \R \to \R$ with bounded derivatives. In this regard, we may view Lemma \ref{order-reduction} as a generalization of \cite[equations (3.2) and (4.11)]{MR3331916}.
\end{remark}

We prove Lemma \ref{order-reduction} in Section \ref{S3:Proofs}. Using Lemma \ref{order-reduction}, we obtain an analogue of \cite[Lemma 2.4]{MR3331916} in the following form.

\begin{lemma}\label{bounded-bilinear-form}
Let $f_1, f_2: \R^d \to \R$ be smooth functions with bounded derivatives. Fix $p \in \Z_+$ and assume that $\alpha_1, \alpha_2, \beta_1, \beta_2, K_1, K_2 \in \Z^d_+$ with
\begin{equation}\label{bounded-bilinear-form-condition}
|\alpha_1| + |\beta_1| + |\alpha_2| + |\beta_2|+ 1_{\{0\}}(|K_1|) + 1_{\{0\}}(|K_2|) \leq 4p.
\end{equation}
Then, the following real valued bilinear form on $\Sc(\R^d)\times\Sc(\R^d)$ defined by 
  \[(\phi,\psi)\mapsto \inpr[0]{(\partial^{K_1}f_1) x^{\alpha_1}\partial^{\beta_1}\phi}{(\partial^{K_2}f_2) x^{\alpha_2}\partial^{\beta_2}\psi}\]
  is bounded in $\|\cdot\|_p$.
 \end{lemma}

When $|K_1| > 0$ for some multi-index $K_1$, then $\partial^{K_1}f_1$ is bounded as per the assumptions of Lemma \ref{bounded-bilinear-form}. When $|K_1| = 0$, then $\partial^{K_1}f_1 = f_1$ has linear growth. We take into account both these cases in \eqref{bounded-bilinear-form-condition}. We prove Lemma \ref{bounded-bilinear-form} in Section \ref{S3:Proofs}. Motivated by this result, we make the following definition.

\begin{definition}[Good terms and FSGT]\label{good-terms}
Let $f_1, f_2: \R^d \to \R$ be smooth functions with bounded derivatives and let $p \in \Z_+$. Assume that $\alpha_1, \alpha_2, \beta_1, \beta_2, K_1, K_2 \in \Z^d_+$ such that \eqref{bounded-bilinear-form-condition} holds. In this case, we refer to terms of the form $\inpr[0]{(\partial^{K_1}f_1) x^{\alpha_1}\partial^{\beta_1}\phi}{(\partial^{K_2}f_2) x^{\alpha_2}\partial^{\beta_2}\psi}$ as good terms. Moreover, a finite sum of good terms shall be abbreviated to `FSGT'.
 \end{definition}

Under the assumptions of Theorem \ref{variable-coeff-monotonicity}, we show that the left hand side of \eqref{Monotoniticity-inequality} is an FSGT and hence the proof of Theorem \ref{variable-coeff-monotonicity} follows from Lemma \ref{bounded-bilinear-form}.

\subsection{Application to Stochastic PDEs}\label{S2:Applications-SPDE}

Let the operators $L$ and $A$ be as stated in \eqref{L-A}. Choose positive integers $p$ and $q$ with $q \leq p - 1$. By Proposition \ref{L-A-bound}, $L, A_1, \cdots, A_d:\Sc_p \to \Sc_q$ are bounded linear operators. Then, using density arguments, we have the following corollary of Theorem \ref{variable-coeff-monotonicity}.

\begin{corollary}\label{Mon-ineq-p-q}
Let $p, q, L$ and $A$ be as stated above. Then
\begin{equation}\label{Monotoniticity-inequality-p-q}
2\inpr[q]{\phi}{L\phi} + \|A\phi\|_{HS(q)}^2\leq C\|\phi\|^2_q, \forall \phi\in \Sc_p(\R^d)
\end{equation}
where $C = C(d, p, q, \sigma, f_1, f_0, h_0) > 0$ is a positive constant not depending on $\phi$. 
\end{corollary}

On a given filtered probability space $(\Omega, \F, (\F_t)_t, \pp)$ satisfying the usual conditions, consider the existence and uniqueness of strong solutions of the following linear SPDE
\begin{equation}\label{linear-spde}
    \begin{cases}
       \displaystyle
       dY_t = L Y_t\, dt + A Y_t. \, dB_t\\
       \displaystyle
        Y_0 = \phi
    \end{cases}
\end{equation}
where $\phi$ is an $\Sc_p(\R^d)$ valued $\F_0$ measurable random variable and $A Y_t. \, dB_t = \sum_{i = 1}^d A_i Y_t. \, dB_t$. Using Corollary \ref{Mon-ineq-p-q} and \cite[Theorem 1]{MR2479730}, we get the next result.

\begin{theorem}\label{exist-unique-soln-spde}
The SPDE \eqref{linear-spde} has an $\Sc_p(\R^d)$ valued strong solution $\{Y_t\}_t$ such that the equality a.s.
\[Y_t = \phi + \int_0^t L Y_s\, ds + \int_0^t A Y_s. \, dB_s, \forall t \geq 0\]
holds in $\Sc_q(\R^d)$. The solution is pathwise unique. Moreover, if $\phi$ is $\Sc(\R^d)$ valued, then so is $\{Y_t\}_t$.
\end{theorem}

Proof of Theorem \ref{exist-unique-soln-spde} is discussed in Section \ref{S3:Proofs}.

In particular, consider $(L, A)$ replaced by the adjoint form $(L^\ast, A^\ast)$ as in \eqref{L-A-adjoint}. Note that $(L^\ast, A^\ast)$ can be written in the form of $(L, A)$ by replacing $\sigma$ by $-\sigma$ and taking 
\begin{align*}
h_j &= \sum_{i = 1}^d \partial_i \sigma_{ij},\\
f_j &= \frac{1}{2}\sum_{i=1}^d  \left\{\partial_i \left(( \sigma \sigma^t )_{ij}\right) + \partial_i \left(( \sigma \sigma^t )_{ji}\right)\right\} - \sum_{j=1}^d b_j,\\
g &= \frac{1}{2}\sum_{i,j=1}^d  \partial^2_{ij} \left(( \sigma \sigma^t )_{ij}\right) - \sum_{j=1}^d \partial_j b_j.
\end{align*}
Hence, as a corollary to Theorem \ref{exist-unique-soln-spde}, we have the next result.

\begin{corollary}\label{exist-unique-adjoint}
Let $\sigma_{ij}: \R^d \to \R, i, j  = 1, \cdots, d$ be bounded functions with bounded derivatives, and $b_j: \R^d \to \R, j  = 1, \cdots, d$ have bounded derivatives. Let $L^\ast$ and $A^\ast$ be given by \eqref{L-A-adjoint}. Consider the linear SPDE
\begin{equation}\label{linear-spde-adjoint}
    \begin{cases}
       \displaystyle
       dY_t = L^\ast Y_t\, dt + A^\ast Y_t .\, dB_t\\
       \displaystyle
        Y_0 = \phi
    \end{cases}
\end{equation}
where $\phi$ is an $\Sc_p(\R^d)$ valued $\F_0$ measurable random variable. This SPDE has an $\Sc_p(\R^d)$ valued strong solution $\{Y_t\}_t$ such that the equality a.s.
\[Y_t = \phi + \int_0^t L^\ast Y_s\, ds + \int_0^t A^\ast Y_s. \, dB_s, \forall t \geq 0\]
holds in $\Sc_q(\R^d)$. The solution is pathwise unique. Moreover, if $\phi$ is $\Sc(\R^d)$ valued, then so is $\{Y_t\}_t$.
\end{corollary}

When the initial condition $\phi$ is a compactly supported distribution, the existence of strong solutions for the SPDE \eqref{linear-spde-adjoint} was proved in \cite[Theorem 3.3]{MR2373102}.

\begin{remark}[Finite Dimensional Realization (see \cite{MR3483747}, see also \cite{MR2373102})]
Let $\phi \in \Sc(\R^d)$. Then, by \cite[Example 2.3]{MR4117986}, the unique solution $\{Y_t\}_t$ in Corollary \ref{exist-unique-adjoint} has a representation in terms of $\{X_t^x\}_t$, the strong solution of an associated finite dimensional stochastic differential equation 
\[dX_t = \sigma(X_t).\, dB_t + b(X_t)\, dt, t > 0, \quad X_0 = x.\]
Here, $Y_t$ is given by $\int_{\R^d} \phi(x)\, \delta_{X^x_t}\, dx$, since for any $\psi \in \Sc(\R^d)$, we have \[\inpr{Y_t}{\psi} = \int_{\R^d} \phi(x)\, \psi(X^x_t)\, dx.\]
\end{remark}

\subsection{Application to PDEs}\label{S2:Applications-PDE}

We continue with $L$ and $p$ as in Subsection \ref{S2:Applications-SPDE}. Along the lines of \cite[Section 4]{MR2373102}, we can obtain probabilistic representation of solutions to the PDEs of the form
\begin{equation}\label{linear-pde}
    \begin{cases}
       \displaystyle
       du_t = L u_t\, dt, t > 0\\
       \displaystyle
        u_0 = \phi
    \end{cases}
\end{equation}
with deterministic initial condition $\phi \in \Sc_p(\R^d)$.

\begin{corollary}\label{exist-unique-soln-pde}
Let $\{Y_t\}_t$ be the solution to the SPDE \eqref{linear-spde} obtained in Theorem \ref{exist-unique-soln-spde} with 
\begin{enumerate}[label=(\roman*)]
    \item the operator $A$ is given by $\sigma$ the same as in $L$ and some $h_0$ a bounded smooth function with bounded derivatives, and
    \item $\phi$ the same as in \eqref{linear-pde}.
\end{enumerate}
Then $u_t = \Exp Y_t, \forall t \geq 0$ is the unique $\Sc_p(\R^d)$ valued solution to the PDE \eqref{linear-pde} in the sense that the equality
\[u_t = \phi + \int_0^t L u_s\, ds, \forall t \geq 0\]
holds in $\Sc_q(\R^d)$. In particular, if $\phi \in \Sc(\R^d)$, then $u_t \in \Sc(\R^d), \forall t \geq 0$.
\end{corollary}

Now, we consider the special case where the operator $L$ is in the adjoint form as in \eqref{L-A-adjoint}, denoted by $L^\ast$. In this case, Corollary \ref{exist-unique-soln-pde} has the following version.

\begin{corollary}\label{linear-adjoint-pde-result}
Let $\{Y_t\}_t$ be the solution to the SPDE \eqref{linear-spde-adjoint} obtained in Corollary \ref{exist-unique-adjoint} with the operator $A^\ast$ is given by $\sigma$ the same as in $L^\ast$. Consider a PDE of the form
\begin{equation}\label{linear-adjoint-pde}
    \begin{cases}
       \displaystyle
       du_t = L^\ast u_t\, dt, t > 0\\
       \displaystyle
        u_0 = \phi
    \end{cases}
\end{equation}
with deterministic initial condition $\phi \in \Sc_p(\R^d)$. Then $u_t = \Exp Y_t, \forall t \geq 0$ is the unique $\Sc_p(\R^d)$ valued solution to the PDE \eqref{linear-adjoint-pde} in the sense that the equality
\[u_t = \phi + \int_0^t L^\ast u_s\, ds, \forall t \geq 0\]
holds in $\Sc_q(\R^d)$. In particular, if $\phi \in \Sc(\R^d)$, then $u_t \in \Sc(\R^d), \forall t \geq 0$.
\end{corollary}

When the initial condition $\phi$ is a compactly supported distribution and the operator $L$ is in the adjoint form as in \eqref{L-A-adjoint}, the existence and uniqueness of solutions for the PDE \eqref{linear-pde} was proved in \cite[Theorems 4.3 and 4.4]{MR2373102}. The proof of Corollary \ref{exist-unique-soln-pde} is given in Section \ref{S3:Proofs}. Corollary \ref{linear-adjoint-pde-result} follows from Corollary \ref{exist-unique-soln-spde} and we do not provide a separate proof for brevity.

\section{Proofs of the results}\label{S3:Proofs}

\begin{proof}[Proof of Lemma \ref{order-reduction}]
We discuss a proof for the one dimension case only. The general $d$-dimensional proof follows along similar lines with notational changes. Observe that
 \begin{align*}
     \inpr[0]{f\partial^\beta\phi}{\partial^{\beta+(2k-1)}\phi}=&(-1)^{2k-1}\inpr[0]{\partial^{2k-1}(f\partial^\beta\phi)}{\partial^\beta\phi}\\
 =&-\inpr[0]{\sum_{r=0}^{2k-1}\binom{2k-1}{r}\partial^r f\partial^{2k-1-r}(\partial^\beta\phi)}{\partial^\beta\phi}\\
 =&-\sum_{r=0}^{2k-1}\binom{2k-1}{r}\inpr[0]{\partial^r f\partial^{2k-1-r}(\partial^\beta\phi)}{\partial^\beta\phi}\\
     =&-\inpr[0]{ f\partial^{\beta+(2k-1)}\phi)}{\partial^\beta\phi}-\sum_{r=1}^{2k-1}\binom{2k-1}{r}\inpr[0]{\partial^r f\partial^{\beta+(2k-1)-r}\phi}{\partial^\beta\phi}\\
     =&-\inpr[0]{f\partial^\beta\phi}{\partial^{\beta+(2k-1)}\phi)}-\sum_{r=1}^{2k-1}\binom{2k-1}{r}\inpr[0]{\partial^r f\partial^{\beta+(2k-1)-r}\phi}{\partial^\beta\phi}
\end{align*}
and hence
\[\inpr[0]{f\partial^\beta\phi}{\partial^{\beta+(2k-1)}\phi} = -\frac{1}{2}\sum_{r=1}^{2k-1}\binom{2k-1}{r}\inpr[0]{\partial^r f\partial^{\beta+(2k-1)-r}\phi}{\partial^\beta\phi}.\]
This completes the proof.
\end{proof}

 \begin{proof}[Proof of Lemma \ref{bounded-bilinear-form}]
 We discuss a proof for the one dimension case only. The general $d$-dimensional proof follows along similar lines with notational changes.
 
 First, we consider the case that $\alpha_i+\beta_i+1_{\{0\}}(K_i)\leq 2p\;\forall i=1,2$. In this proof, we shall refer to this condition as case I. This condition implies \eqref{bounded-bilinear-form-condition}. Observe that
\begin{align*}
    \left|{\inpr{\partial^{K_1}f_1 .x^{\alpha_1}\partial^{\beta_1}\phi}{\partial^{K_2}f_2. x^{\alpha_2}\partial^{\beta_2}\psi}}_0\right|\leq&\|\partial^{K_1}f_1 .x^{\alpha_1}\partial^{\beta_1}\phi\|_0\|\partial^{K_2}f_2. x^{\alpha_2}\partial^{\beta_2}\psi\|_0.
\end{align*}

Now, suppose $K_i > 0$ for some $i$. Then, we have $\alpha_i + \beta_i \leq 2p$. Using the boundedness of the derivatives of $f_i$, we conclude
\[\|\partial^{K_i}f_i .x^{\alpha_i}\partial^{\beta_i}\phi\|_0 \leq \sup|\partial^{K_i}f_i|\|x^{\alpha_i}\partial^{\beta_i}\phi\|_0
    \leq C_f\|\phi\|_p,\]
for some constant $C_f > 0$ depending on $f$.

If $K_i = 0$, then $\partial^{K_i}f_i = f_i$ has linear growth. In this case, we have $\alpha_i + \beta_i \leq 2p - 1$ and hence
\[\|\partial^{K_i}f_i .x^{\alpha_i}\partial^{\beta_i}\phi\|_0 \leq C_f \|(1+|x|)x^{\alpha_i}\partial^{\beta_i}\phi\|_0 
    \leq 2C_f \|\phi\|_p,\]
for some constant $C_f > 0$ depending on $f$.

Therefore, the condition $\alpha_i+\beta_i+1_{\{0\}}(K_i)\leq 2p\;\forall i=1,2$ implies 
\begin{align*}
    \left|{\inpr{\partial^{K_1}f_1 .x^{\alpha_1}\partial^{\beta_1}\phi}{\partial^{K_2}f_2. x^{\alpha_2}\partial^{\beta_2}\psi}}_0\right|
    \leq&C\|\phi\|_p\|\psi\|_p
\end{align*}

Now, we consider the remaining case when \eqref{bounded-bilinear-form-condition} holds. We have $\alpha_i+\beta_i+1_{\{0\}}(K_i) > 2p$, for either $i=1$ or $i=2$, but not both at the same time. In this proof, we refer to this case as case II. Without loss of generality, we take $\alpha_1+\beta_1+1_{\{0\}}(K_1) > 2p$. There exists non-negative integers $n_1,n_2,n_3$ less or equal to $\alpha_2, \beta_2,1_{\{0\}}(k_1)$, respectively, such that  $(\alpha_1-n_1)+(\beta_1-n_2)+(1_{\{0\}}(k_1)-n_3)= 2p$. Then, 
\begin{equation}\label{rewrite-case-II}
\begin{split}
     &\left|{\inpr{\partial^{K_1}f_1 .x^{\alpha_1}\partial^{\beta_1}\phi}{\partial^{K_2}f_2. x^{\alpha_2}\partial^{\beta_2}\psi}}_0\right|\\
     =&\left|{\inpr{\partial^{\beta_1}\phi}{\partial^{K_1}f_1\partial^{K_2}f_2. x^{\alpha_1+\alpha_2}\partial^{\beta_2}\psi}}_0\right|\\
     =&\left|{\inpr{\partial^{\beta_1-n_2}\phi}{\partial^{n_2}\left(\partial^{K_1}f_1\partial^{K_2}f_2. x^{\alpha_1+\alpha_2}\partial^{\beta_2}\psi\right)}}_0\right|\\
     =&\left|{\inpr{\partial^{\beta_1-n_2}\phi}{\sum_{0\leq r_1\leq n_2}\binom{n_2}{r_1}\partial^{r_1}\left(\partial^{K_1}f_1\partial^{K_2}f_2\right)\partial^{n_2-r_1}\left( x^{\alpha_1+\alpha_2}\partial^{\beta_2}\psi\right)}}_0\right|\\
       =&\left|\sum_{0\leq r_1\leq n_2}\binom{n_2}{r_1}{\inpr{\partial^{\beta_1-n_2}\phi}{\partial^{r_1}\left(\partial^{K_1}f_1\partial^{K_2}f_2\right)\partial^{n_2-r_1}\left( x^{\alpha_1+\alpha_2}\partial^{\beta_2}\psi\right)}}_0\right|\\
              \leq&\sum_{0\leq r_1\leq n_2}\binom{n_2}{r_1}\left|{\inpr{\partial^{\beta_1-n_2}\phi}{\partial^{r_1}\left(\partial^{K_1}f_1\partial^{K_2}f_2\right)\partial^{n_2-r_1}\left( x^{\alpha_1+\alpha_2}\partial^{\beta_2}\psi\right)}}_0\right|\\
                  \leq&\sum_{ r_1=0}^{n_2}\sum_{r_2=0}^{n_2-r_1}\binom{n_2}{r_1}\binom{n_2-r_1}{r_2}\left|{\inpr{\partial^{\beta_1-n_2}\phi}{\partial^{r_1}\left(\partial^{K_1}f_1\partial^{K_2}f_2\right)\partial^{r_2}\left( x^{\alpha_1+\alpha_2}\right)\partial^{n_2-r_1-r_2}\left(\partial^{\beta_2}\psi\right)}}_0\right|\\
                  \leq& C\sum_{ r_1=0}^{n_2}\sum_{r_2=0}^{n_2-r_1}\left|{\inpr{\partial^{\beta_1-n_2}\phi}{\partial^{r_1}\left(\partial^{K_1}f_1\partial^{K_2}f_2\right)\partial^{r_2}\left( x^{\alpha_1+\alpha_2}\right)\partial^{n_2-r_1-r_2}\left(\partial^{\beta_2}\psi\right)}}_0\right|\\
                  \leq& C\sum_{ r_1=0}^{n_2}\sum_{r_2=0}^{n_2-r_1}\sum_{r_3=0}^{r_1}\left|{\inpr{\partial^{\beta_1-n_2}\phi}{\partial^{K_1+r_3}f_1\partial^{K_2+r_1-r_3}f_2\left( x^{\alpha_1+\alpha_2-r_2 }\right)\left(\partial^{\beta_2+n_2-r_1-r_2}\psi\right)}}_0\right|.
\end{split}
\end{equation}
Now, for each fixed $r_1, r_2$ and $r_3$ as considered in the above sum, we note that the term
\begin{align*}
&\left|{\inpr{\partial^{\beta_1-n_2}\phi}{\partial^{K_1+r_3}\sigma_1\partial^{K_2+r_1-r_3}\sigma_2\left( x^{\alpha_1+\alpha_2-r_2 }\right)\left(\partial^{\beta_2+n_2-r_1-r_2}\psi\right)}}_0\right|\\
&=\left|{\inpr{x^{\alpha_1-n_1}\left(\partial^{K_1+r_3}\sigma_1\right)^{1-n_3}\partial^{\beta_1-n_2}\phi}{x^{n_1+\alpha_2-r_2 }\left(\partial^{K_1+r_3}\sigma_1\right)^{n_3}\partial^{K_2+r_1-r_3}\sigma_2\left( \partial^{\beta_2+n_2-r_1-r_2}\psi\right)}}_0\right|,
\end{align*}
and hence, this term falls in case I for which we already have the required bound. Then, from \eqref{rewrite-case-II}, we have
\begin{align*}
    \left|{\inpr{\partial^{K_1}f_1 .x^{\alpha_1}\partial^{\beta_1}\phi}{\partial^{K_2}f_2. x^{\alpha_2}\partial^{\beta_2}\psi}}_0\right|
    \leq&C\|\phi\|_p\|\psi\|_p
\end{align*}
even in case II. This completes the proof.
\end{proof}

\begin{proof}[Proof of Theorem \ref{variable-coeff-monotonicity}]
we have
\begin{align*}
    \| A \phi \|^2_{HS(p)} = & \sum_{i=1}^d\|A_i \phi\|^2_p\\
    =& \sum_{i,j,k=1}^d \inpr[p]{\sigma_{ik} \partial_i \phi }{\sigma_{jk} \partial_j \phi } + 2 \sum_{i,j=1}^d \inpr[p]{\sigma_{ji} \partial_j \phi }{h_i \phi } + \sum_{i=1}^d\inpr[p]{h_i \phi}{h_i\phi}
\end{align*}
and \[ \inpr[p]{\phi}{L \phi}= \sum_{i,j,k=1}^d \inpr[p]{ \phi }{\sigma_{ik} \sigma_{jk} \partial_{ij}^2 \phi } + \sum_{i=1}^d \inpr[p]{ \phi }{f_i \partial_i \phi} + \inpr[p]{ \phi}{g \phi}.\]
Then,
\begin{equation}\label{Mon-ineq-LHS}
\begin{split}
   \inpr[p]{\phi}{L \phi} +  \| A \phi \|^2_{HS(p)} = & \sum_{i,j,k=1}^d \left(\inpr[p]{ \phi }{\sigma_{ik} \sigma_{jk} \partial_{ij}^2 \phi } + \inpr[p]{\sigma_{ik} \partial_i \phi }{\sigma_{jk} \partial_j \phi } \right) + 2 \sum_{i,j=1}^d \inpr[p]{\sigma_{ji} \partial_j \phi }{h_i \phi }\\
   &+ \sum_{ i = 1 }^d  \left ( \inpr[p]{ \phi }{f_i \partial_i \phi} + \inpr[p]{h_i \phi}{h_i \phi} \right ) + \inpr[p]{ \phi}{g \phi}.
\end{split}
\end{equation}
We now look at the terms on the right hand side of \eqref{Mon-ineq-LHS} separately. First, we look at the terms involving derivatives of order two.

\begin{align*}
    &\inpr[p]{ \phi }{\sigma_{ik} \sigma_{jk} \partial_{ij}^2 \phi }\\
    =& \sum_{|\alpha| + |\beta| \leq 2p}\inpr[]{x ^ \alpha \partial ^ \beta \phi}{x ^ \alpha \partial ^ \beta (\sigma_{ik} \sigma_{jk} \partial_{ij}^2 \phi)}\\
    =&\sum_{|\alpha| + |\beta| \leq 2p} \sum_{r \leq \beta} \binom{\beta}{r}\inpr[]{x ^ \alpha \partial ^ \beta \phi}{x ^ \alpha \partial ^ r (\sigma_{ik} \sigma_{jk} )\partial ^ {\beta + e_i + e_j - r} \phi}\\
     =&\sum_{|\alpha| + |\beta| \leq 2p} \inpr[]{x ^ \alpha \partial ^ \beta \phi}{x ^ \alpha \sigma_{ik} \sigma_{jk} \partial ^ {\beta + e_i + e_j } \phi}+ \sum_{|\alpha| + |\beta| \leq 2p}\sum ^d_{q=1}  
 \beta_q \inpr[]{x ^ \alpha \partial ^ \beta \phi}{x ^ \alpha \partial_q (\sigma_{ik} \sigma_{jk} )\partial ^ {\beta + e_i + e_j - e_q} \phi} \\
     &+\sum_{|\alpha| + |\beta| \leq 2p} \sum_{r \leq \beta, |r|=2} \binom{\beta}{r}\inpr[]{x ^ \alpha \partial ^ \beta \phi}{x ^ \alpha \partial ^ r (\sigma_{ik} \sigma_{jk} )\partial ^ {\beta + e_i + e_j - r} \phi}\\
      =&\sum_{|\alpha| + |\beta| \leq 2p} \inpr[]{ \partial ^ \beta \phi}{x ^ {2\alpha} \sigma_{ik} \sigma_{jk} \partial ^ {\beta + e_i + e_j } \phi}\\
      &+ \sum_{|\alpha| + |\beta| \leq 2p}\sum ^d_{q=1}  
 \frac{\beta_q}{2} \left ( -\inpr[]{\partial ^ \beta \phi}{x ^ {2\alpha} \partial_{qj}^2 (\sigma_{ik} \sigma_{jk} )\partial ^ {\beta + e_i  - e_q} \phi} -\inpr[]{\partial ^ \beta \phi}{x ^ {2\alpha} \partial_{qi}^2 (\sigma_{ik} \sigma_{jk} )\partial ^ {\beta + e_j  - e_q} \phi}  \right ) \\
     &+ \sum_{|\alpha| + |\beta| \leq 2p}\sum ^d_{q=1}  
  \frac{\beta_q}{2}  \inpr[]{\partial ^ \beta \phi}{x ^ {2\alpha} \partial_{q}^2 (\sigma_{ik} \sigma_{jk} )\partial ^ {\beta + e_i +e_j  - 2e_q} \phi}\\
 &+\sum_{|\alpha| + |\beta| \leq 2p} \sum_{r \leq \beta, |r|=2} \binom{\beta}{r}\inpr[]{x ^ \alpha \partial ^ \beta \phi}{x ^ \alpha \partial ^ r (\sigma_{ik} \sigma_{jk} )\partial ^ {\beta + e_i + e_j - r} \phi}\\
\end{align*}

The other term with a cumulative order of the derivatives being two is as follows.

\begin{align*}
    &\inpr[]{\sigma_{ik} \partial_i \phi }{\sigma_{jk} \partial_j \phi }\\
    = & \sum_{|\alpha| + |\beta| \leq 2p}\inpr[]{x ^ \alpha \partial ^ \beta(\sigma_{ik}\partial_i\phi)}{x ^ \alpha \partial ^ \beta(\sigma_{jk}\partial_j\phi)} \\
    = & \sum_{|\alpha| + |\beta| \leq 2p} \sum_{r_1, r_2 \leq \beta}  \binom{\beta}{r_1} \binom{\beta}{r_2}\inpr[]{x ^ \alpha \partial ^ {r_1} \sigma_{ik} \partial ^ {\beta + e_i - r_1} \phi}{x ^ \alpha \partial ^ {r_2} \sigma_{jk} \partial ^ {\beta + e_j - r_2}\phi}\\
    = & \sum_{|\alpha| + |\beta| \leq 2p} \inpr[]{x ^ \alpha  \sigma_{ik} \partial ^ {\beta + e_i } \phi}{x ^ \alpha  \sigma_{jk} \partial ^ {\beta + e_j }\phi} + \sum_{|\alpha| + |\beta| \leq 2p} \sum_{q=1} ^ d  \beta_q \inpr[]{x ^ \alpha \partial _ q \sigma_{ik} \partial ^ {\beta + e_i - 
 e_q} \phi}{x ^ \alpha  \sigma_{jk} \partial ^ {\beta + e_j }\phi}\\
 +&\sum_{|\alpha| + |\beta| \leq 2p} \sum_{q=1} ^ d  \beta_q \inpr[]{x ^ \alpha  \sigma_{ik} \partial ^ {\beta + e_i }\phi}{x ^ \alpha \partial _ q \sigma_{jk} \partial ^ {\beta + e_j - 
 e_q} \phi}\\
 +&\sum_{|\alpha| + |\beta| \leq 2p} \sum_{r_1 \leq \beta, |r_1|=2}  \binom{\beta}{r_1} \inpr[]{x ^ \alpha \partial ^ {r_1} \sigma_{ik} \partial ^ {\beta + e_i - r_1} \phi}{x ^ \alpha  \sigma_{jk} \partial ^ {\beta + e_j }\phi}\\
 +&\sum_{|\alpha| + |\beta| \leq 2p} \sum_{r_2 \leq \beta, |r_2|=2}  \binom{\beta}{r_2} \inpr[]{x ^ \alpha  \sigma_{ik} \partial ^ {\beta + e_i } \phi}{x ^ \alpha \partial ^ {r_2} \sigma_{jk} \partial ^ {\beta + e_j - r_2}\phi}\\
  =&  \sum_{|\alpha| + |\beta| \leq 2p} \left ( -\inpr[]{  \partial ^ {\beta  } \phi}{x ^ {2\alpha} \sigma_{ik}   \sigma_{jk} \partial ^ {\beta + e_j+ e_i }\phi} + \frac{1}{2}\inpr[]{  \partial ^ {\beta  } \phi}{x ^ {2\alpha}  \partial_{ij}^2 (\sigma_{ik}   \sigma_{jk}) \partial ^ {\beta }\phi} \right ) \\
  +&\sum_{|\alpha| + |\beta| \leq 2p} \sum_{q=1} ^ d  \frac{\beta_q}{2}  \left ( -\inpr[]{\partial ^ {\beta }\phi}{x ^ {2\alpha} \sigma_{jk}\partial^2 _ {qj} \sigma_{ik} \partial ^ {\beta + e_i - 
 e_q} \phi} - \inpr[]{\partial ^ {\beta }\phi}{x ^ {2\alpha} \sigma_{jk} \partial^2 _ {q} \sigma_{ik} \partial ^ {\beta + e_i+e_j - 
 2e_q} \phi}  \right )\\
 +&\sum_{|\alpha| + |\beta| \leq 2p} \sum_{q=1} ^ d  \frac{\beta_q}{2}  \inpr[]{\partial ^ {\beta }\phi}{x ^ {2\alpha} \sigma_{jk}\partial^2 _ {qi} \sigma_{ik} \partial ^ {\beta + e_i - 
 e_q} \phi}\\
 +&\sum_{|\alpha| + |\beta| \leq 2p} \sum_{q=1} ^ d  \frac{\beta_q}{2}  \left ( -\inpr[]{\partial ^ {\beta }\phi}{x ^ {2\alpha} \sigma_{ik}\partial^2 _ {qi} \sigma_{jk} \partial ^ {\beta + e_j - 
 e_q} \phi} - \inpr[]{\partial ^ {\beta }\phi}{x ^ {2\alpha} \sigma_{ik} \partial^2 _ {q} \sigma_{jk} \partial ^ {\beta + e_i+e_j - 
 2e_q} \phi}  \right )\\
 +&\sum_{|\alpha| + |\beta| \leq 2p} \sum_{q=1} ^ d  \frac{\beta_q}{2}  \inpr[]{\partial ^ {\beta }\phi}{x ^ {2\alpha} \sigma_{ik}\partial^2 _ {qj} \sigma_{jk} \partial ^ {\beta + e_j - 
 e_q} \phi}\\
 -& \sum_{|\alpha| + |\beta| \leq 2p} \sum_{r_2 \leq \beta, |r|=2}  \binom{\beta}{r} \inpr[]{  \partial ^ {\beta  } \phi}{x ^ {2\alpha} \partial ^ {r}(\sigma_{ik} \sigma_{jk}) \partial ^ {\beta + e_i + e_j - r}\phi}\\
 =&  \sum_{|\alpha| + |\beta| \leq 2p} \left ( -\inpr[]{  \partial ^ {\beta  } \phi}{x ^ {2\alpha} \sigma_{ik}   \sigma_{jk} \partial ^ {\beta + e_j+ e_i }\phi} + \frac{1}{2}\inpr[]{  \partial ^ {\beta  } \phi}{x ^ {2\alpha}  \partial_{ij}^2 (\sigma_{ik}   \sigma_{jk}) \partial ^ {\beta }\phi} \right ) \\
  +&\sum_{|\alpha| + |\beta| \leq 2p} \sum_{q=1} ^ d  \frac{\beta_q}{2}  \left ( -\inpr[]{\partial ^ {\beta }\phi}{x ^ {2\alpha} \sigma_{jk}\partial^2 _ {qj} \sigma_{ik} \partial ^ {\beta + e_i - 
 e_q} \phi} - \inpr[]{\partial ^ {\beta }\phi}{x ^ {2\alpha} \partial^2 _ {q} ( \sigma_{jk}  \sigma_{ik} ) \partial ^ {\beta + e_i+e_j - 
 2e_q} \phi}  \right )\\
 +&\sum_{|\alpha| + |\beta| \leq 2p} \sum_{q=1} ^ d  \frac{\beta_q}{2} \inpr[]{\partial ^ {\beta }\phi}{x ^ {2\alpha} \sigma_{jk}\partial^2 _ {qi} \sigma_{ik} \partial ^ {\beta + e_i - 
 e_q} \phi}\\
 +& \sum_{|\alpha| + |\beta| \leq 2p} \sum_{q=1} ^ d  -\frac{\beta_q}{2}  \inpr[]{\partial ^ {\beta }\phi}{x ^ {2\alpha} \sigma_{ik}\partial^2 _ {qi} \sigma_{jk} \partial ^ {\beta + e_j - 
 e_q} \phi}  \\
 +&\sum_{|\alpha| + |\beta| \leq 2p} \sum_{q=1} ^ d  \frac{\beta_q}{2}  \inpr[]{\partial ^ {\beta }\phi}{x ^ {2\alpha} \sigma_{ik}\partial^2 _ {qj} \sigma_{jk} \partial ^ {\beta + e_j - 
 e_q} \phi}\\
 -& \sum_{|\alpha| + |\beta| \leq 2p} \sum_{r_2 \leq \beta, |r|=2}  \binom{\beta}{r} \inpr[]{  \partial ^ {\beta  } \phi}{x ^ {2\alpha} \partial ^ {r}(\sigma_{ik} \sigma_{jk}) \partial ^ {\beta + e_i + e_j - r}\phi}
\end{align*}

We now look at the terms involving only derivatives of the first order. First,
\begin{align*}
    &\inpr[p]{\sigma_{ji} \partial_j \phi }{h_i \phi }\\
    = & \sum_{|\alpha| + |\beta| \leq 2p} \inpr[]{x ^ \alpha \partial ^ \beta(\sigma_{ji} \partial_j \phi)}{x ^ \alpha \partial ^ \beta(h_i \phi)}\\
    = & \sum_{|\alpha| + |\beta| \leq 2p} \sum_{r_1 \leq \beta, r_2 \leq \beta} \binom{\beta}{r_1}\binom{\beta}{r_2} \inpr[]{x ^ \alpha \partial ^ {r_1} \sigma_{ji} \partial ^ {\beta + e_j - r_1} \phi)}{x ^ \alpha \partial ^ {r_2} h_i \partial ^ {\beta - r_2} \phi)}\\
    = & \sum_{|\alpha| + |\beta| \leq 2p}  \inpr[]{x ^ \alpha \sigma_{ji} \partial ^ {\beta +e_j} \phi)}{x ^ \alpha  h_i \partial ^ \beta  \phi} +  \sum_{|\alpha| + |\beta| \leq 2p} \sum_{q = 1} ^ d  \beta_q \inpr[]{x ^ \alpha \partial _q \sigma_{ji} \partial ^ {\beta + e_j - e_q} \phi}{x ^ \alpha  h_i \partial ^ \beta  \phi}\\
    +& \sum_{|\alpha| + |\beta| \leq 2p} \sum_{q = 1} ^ d \beta_q \inpr[]{x ^ \alpha  \sigma_{ji} \partial ^ {\beta +e_j} \phi}{x ^ \alpha \partial _q h_i \partial ^ {\beta - e_q} \phi}\\
    +& \sum_{|\alpha| + |\beta| \leq 2p} \sum_{r \leq \beta, |r| = 2 } \binom{\beta}{r} \inpr[]{x ^ \alpha \sigma_{ji} \partial ^ {\beta +e_j} \phi}{x ^ \alpha \partial ^ {r_2} h_i \partial ^ {\beta - r_2} \phi} + FSGT\\
    = & \sum_{|\alpha| + |\beta| \leq 2p} - \frac{1}{2} \left ( \inpr[]{ \partial ^ \beta  \phi}{x ^ {2\alpha} 
   \partial_j \left( \sigma_{ji}  h_i \right ) \partial ^ \beta  \phi} + \inpr[]{ \partial ^ \beta  \phi}{2 \alpha_j x ^ {2\alpha - e_j} 
  \sigma_{ji}  h_i \partial ^ \beta  \phi} \right )\\
  +&  \sum_{|\alpha| + |\beta| \leq 2p} \sum_{q = 1} ^ d  \beta_q \inpr[]{ \partial ^ \beta  \phi}{x ^ {2\alpha} \left ( h_i \partial _q \sigma_{ji}- \sigma_{ji} \partial _q h_i \right ) \partial ^ {\beta + e_j - e_q} \phi} + FSGT
\end{align*}

Also,
\begin{align*}
    &\inpr[p]{ \phi }{f_i \partial_i \phi}\\
    = & \sum_{|\alpha| + |\beta| \leq 2p}\inpr[]{x ^ \alpha \partial ^ \beta \phi}{x ^ \alpha \partial ^ \beta(f_i \partial_i \phi)}\\
    = & \sum_{|\alpha| + |\beta| \leq 2p} \sum_{r \leq \beta }\binom{\beta}{r} \inpr[]{x ^ \alpha \partial ^ \beta \phi}{x ^ \alpha \partial ^ r f_i \partial ^ {\beta + e_i - r} \phi}\\
    = & \sum_{|\alpha| + |\beta| \leq 2p} \inpr[]{\partial ^ \beta \phi}{x ^ {2\alpha}  f_i \partial ^ {\beta + e_i } \phi} +  \sum_{|\alpha| + |\beta| \leq 2p} \sum_{q=1}^d \beta_q \inpr[]{ \partial ^ \beta \phi}{x ^ {2\alpha} \partial _q f_i \partial ^ {\beta + e_i - e_q} \phi} + FSGT\\
    = & \sum_{|\alpha| + |\beta| \leq 2p} -\frac{1}{2} \inpr[]{\partial ^ \beta \phi}{x ^ {2\alpha}  \partial_i f_i \partial ^ \beta  \phi} +  \sum_{|\alpha| + |\beta| \leq 2p} \sum_{q=1}^d \beta_q \inpr[]{ \partial ^ \beta \phi}{x ^ {2\alpha} \partial _q f_i \partial ^ {\beta + e_i - e_q} \phi} + FSGT
\end{align*} 

We also have terms which do not involve any derivative on $\phi$. First,
\begin{align*}
     \inpr[p]{h_i \phi}{h_i \phi}= & \sum_{|\alpha| + |\beta| \leq 2p}\inpr[]{x ^ \alpha \partial ^ \beta(h_i\phi)}{x ^ \alpha \partial ^ \beta(h_i\phi)}\\
     = & \sum_{|\alpha| + |\beta| \leq 2p} \sum _{r_1 \leq \beta, r_2 \leq \beta} \binom{\beta}{r_1}\binom{\beta}{r_2} \inpr[]{x ^ \alpha \partial^{r_1} h_i \partial ^ {\beta-r_1} \phi}{x ^ \alpha \partial^{r_2} h_i \partial ^ {\beta-r_2} \phi}\\
     = & \sum_{|\alpha| + |\beta| \leq 2p} \inpr[]{ \partial ^ \beta \phi}{x ^ {2\alpha}  h_i^2 \partial ^ \beta \phi}  + FSGT
\end{align*}

The remaining term is as follows.
\begin{align*}
    \inpr[p]{ \phi}{g \phi}= & \sum_{|\alpha| + |\beta| \leq 2p}\inpr[]{x ^ \alpha \partial ^ \beta \phi}{x ^ \alpha \partial ^ \beta(g \phi)}\\
    = & \sum_{|\alpha| + |\beta| \leq 2p}\inpr[]{\partial ^ \beta \phi}{x ^ {2\alpha} g \partial ^ \beta \phi} + FSGT
\end{align*}

Using all these intermediate expressions, from \eqref{Mon-ineq-LHS}, we have
\begin{align*}
   &\inpr[p]{\phi}{L \phi} +  \| A \phi \|^2_{HS(p)}\\
   =& \sum_{i,j,k} \sum_{|\alpha| + |\beta| \leq 2p}   \frac{1}{2}\inpr[]{  \partial ^ {\beta  } \phi}{x ^ {2\alpha}  \partial_{ij}^2 (\sigma_{ik}   \sigma_{jk}) \partial ^ {\beta } \phi}  \\
  +&\sum_{i,j,k}\sum_{|\alpha| + |\beta| \leq 2p} \sum_{q=1} ^ d  -\beta_q \inpr[]{\partial ^ {\beta }\phi}{x ^ {2\alpha} \sigma_{jk}\partial^2 _ {qj} \sigma_{ik} \partial ^ {\beta + e_i - 
 e_q} \phi}\\
 +& \sum_{i,j,k}\sum_{|\alpha| + |\beta| \leq 2p} \sum_{q=1} ^ d  -\beta_q  \inpr[]{\partial ^ {\beta }\phi}{x ^ {2\alpha} \sigma_{ik} \partial^2 _ {qi} \sigma_{jk} \partial ^ {\beta + e_j - 
 e_q} \phi}\\
  +& 2\sum_{i,j}\sum_{|\alpha| + |\beta| \leq 2p} - \frac{1}{2} \left ( \inpr[]{ \partial ^ \beta  \phi}{x ^ {2\alpha} 
   \partial_j \left( \sigma_{ji}  h_i \right ) \partial ^ \beta  \phi} + \inpr[]{ \partial ^ \beta  \phi}{2 \alpha_j x ^ {2\alpha - e_j} 
  \sigma_{ji}  h_i \partial ^ \beta  \phi} \right )\\
  +& 2\sum_{i,j} \sum_{|\alpha| + |\beta| \leq 2p} \sum_{q = 1} ^ d  \beta_q \inpr[]{ \partial ^ \beta  \phi}{x ^ {2\alpha} \left ( h_i \partial _q \sigma_{ji}- \sigma_{ji} \partial _q h_i \right ) \partial ^ {\beta + e_j - e_q} \phi}\\
  +& \sum_{i=1}^d \left (  \sum_{|\alpha| + |\beta| \leq 2p} -\frac{1}{2} \inpr[]{\partial ^ \beta \phi}{x ^ {2\alpha}  \partial_i f_i \partial ^ \beta  \phi} +  \sum_{|\alpha| + |\beta| \leq 2p} \sum_{q=1}^d \beta_q \inpr[]{ \partial ^ \beta \phi}{x ^ {2\alpha} \partial _q f_i \partial ^ {\beta + e_i - e_q} \phi} \right)\\
  +&\sum_{i=1}^d \sum_{|\alpha| + |\beta| \leq 2p} \inpr[]{ \partial ^ \beta \phi}{x ^ {2\alpha}  h_i^2 \partial ^ \beta \phi} + \sum_{|\alpha| + |\beta| \leq 2p} \inpr[]{ \partial ^ \beta \phi}{x ^ {2\alpha} g \partial ^ \beta \phi}    +   FSGT\\
  =&\sum_{|\alpha| + |\beta| \leq 2p}  \inpr[]{ \partial ^ \beta \phi}{x^{ 2 \alpha} \left ( \frac{1}{2} \sum_{i,j,k} \partial_{i,j} ( \sigma_{jk} \sigma_{ik}) - \sum_{i,j=1} ^ d \partial_j (\sigma_{ji} h_i) + \sum _{i=1}^d \left ( - \frac{1}{2} \partial_i f_i +h_i^2 \right )+ g \right )\partial^\beta \phi}\\
  +& \sum_{|\alpha| + |\beta| \leq 2p} \sum_{j,q=1} ^ d \beta_q   \inpr[]{\partial^ \beta \phi}{x^{2 \alpha} \left (-\sum_{i,k=1} ^ d \sigma_{ik} \partial_{qi} ^ 2 \sigma_{jk}  +2 \sum_{i=1} ^ d \left ( h_i \partial_q \sigma_{ji} - \sigma_{ji} \partial_q h_i \right ) \right ) \partial^{ \beta + e_j 
 - e_q} \phi }\\
  +& \sum_{|\alpha| + |\beta| \leq 2p} \sum_{j,q=1} ^ d \beta_q   \inpr[]{\partial^ \beta \phi}{x^{2 \alpha} \left ( -\sum_{j,k=1} ^ d \sigma_{jk} \partial_{qj} ^ 2 \sigma_{ik}  + \partial_q f_i \right )\partial^{\beta + e_i - e_q} \phi }\\
  +& \sum_{i,j=1} ^ d \inpr[]{\partial^ \beta \phi}{ 2 \alpha_j x ^ {2\alpha - e_j } (\sigma_{ji} h_i) \partial ^ \beta \phi} + FSGT
\end{align*}

Hence the Monotonicity inequality holds.
\end{proof}

\begin{proof}[Proof of Theorem \ref{exist-unique-soln-spde}]
First, assume that $\Exp \|\phi\|^2_p < \infty$. By Corollary \ref{Mon-ineq-p-q}, we have
the following Monotonicity inequality \eqref{Monotoniticity-inequality-p-q}
\[2\inpr[q]{\psi}{L\psi} + \|A\psi\|_{HS(q)}^2\leq C\|\psi\|^2_q, \forall \psi\in \Sc_p(\R^d).\]
As an application of \cite[Theorem 1]{MR2479730}, we get the required existence and pathwise uniqueness in this case. Extension to arbitrary initial condition $\phi$ can be done as in \cite{brajeev-arxiv}.


If $\phi$ is $\Sc(\R^d)$-valued, then it is $\Sc_p(\R^d)$-valued for all $p \geq 0$ and consequently, the corresponding solution $\{Y_t\}_t$ to \eqref{linear-spde} is also $\Sc_p(\R^d)$-valued for all $p \geq 0$. In this case, $\{Y_t\}_t$ is $\Sc(\R^d)$-valued. 
\end{proof}

\begin{proof}[Proof of Corollary \ref{exist-unique-soln-pde}]
Since $\phi$ is deterministic, from \cite[Theorem 1]{MR2479730}, the solution $\{Y_t\}_t$ to the SPDE \eqref{linear-spde} has the property that $\Exp \|Y_t\|_p^2 < \infty, \forall t \geq 0$. The fact that $u_t = \Exp Y_t, \forall t \geq 0$ solves the PDE \eqref{linear-pde} follows from the linearity of $L$. 

To prove the uniqueness, let $\tilde u_t, t \geq 0$ be a solution to the PDE \eqref{linear-pde}. Then, $u_t - \tilde u_t = \int_0^t L(u_s - \tilde u_s)\, ds$ and hence
\begin{align*}
\|u_t - \tilde u_t\|_q^2 &= 2 \int_0^t \inpr[q]{u_s - \tilde u_s}{L(u_s - \tilde u_s)}\, ds\\
&\leq  \int_0^t \left[2\inpr[q]{u_s - \tilde u_s}{L(u_s - \tilde u_s)} + \|A(u_s - \tilde u_s)\|_{HS(q)}^2\right]\, ds.
\end{align*}
The uniqueness then follows by Corollary \ref{Mon-ineq-p-q} and the Gronwall's inequality. 

If $\phi$ is $\Sc(\R^d)$-valued, then it is $\Sc_p(\R^d)$-valued for all $p \geq 0$ and consequently, the corresponding solution $\{u_t\}_t$ to \eqref{linear-pde} is also $\Sc_p(\R^d)$-valued for all $p \geq 0$. In this case, $\{Y_t\}_t$ is $\Sc(\R^d)$-valued. This completes the proof.
\end{proof}

\noindent\textbf{Acknowledgements:} Arvind Kumar Nath acknowledges the support by the University Grants Commission (Government of India) Ph.D research Fellowship. Suprio Bhar acknowledges the support by the Matrics grant MTR/2021/000517 from the Science and Engineering Research Board (Department of Science \& Technology, Government of India).

\bibliographystyle{plain}
\bibliography{ref}

\begin{thebibliography}{10}

\bibitem{MR4600818}
Nacira Agram and Bernt \O~ksendal.
\newblock Stochastic {F}okker-{P}lanck equations for conditional
  {M}c{K}ean-{V}lasov jump diffusions and applications to optimal control.
\newblock {\em SIAM J. Control Optim.}, 61(3):1472--1493, 2023.

\bibitem{MR4583740}
Viorel Barbu and Michael R\"{o}ckner.
\newblock Uniqueness for nonlinear {F}okker-{P}lanck equations and for
  {M}c{K}ean-{V}lasov {SDE}s: the degenerate case.
\newblock {\em J. Funct. Anal.}, 285(4):Paper No. 109980, 37, 2023.

\bibitem{MR1329542}
Richard~F. Bass.
\newblock {\em Probabilistic techniques in analysis}.
\newblock Probability and its Applications (New York). Springer-Verlag, New
  York, 1995.

\bibitem{MR4568882}
Suprio Bhar, Rajeev Bhaskaran, and Barun Sarkar.
\newblock Erratum to `{S}olutions of {SPDE}'s associated with a stochastic
  flow'.
\newblock {\em Potential Anal.}, 58(4):785--786, 2023.

\bibitem{MR3331916}
Suprio Bhar and B.~Rajeev.
\newblock Differential operators on {H}ermite {S}obolev spaces.
\newblock {\em Proc. Indian Acad. Sci. Math. Sci.}, 125(1):113--125, 2015.

\bibitem{MR4117986}
Suprio Bhar, Bhaskaran Rajeev, and Barun Sarkar.
\newblock Solutions of {SPDE}'s associated with a stochastic flow.
\newblock {\em Potential Anal.}, 53(1):203--221, 2020.

\bibitem{MR2674056}
Vladimir Bogachev, Giuseppe Da~Prato, and Michael R\"{o}ckner.
\newblock Existence and uniqueness of solutions for {F}okker-{P}lanck equations
  on {H}ilbert spaces.
\newblock {\em J. Evol. Equ.}, 10(3):487--509, 2010.

\bibitem{MR2857016}
Vladimir Bogachev, Giuseppe Da~Prato, and Michael R\"{o}ckner.
\newblock Existence results for {F}okker-{P}lanck equations in {H}ilbert
  spaces.
\newblock In {\em Seminar on {S}tochastic {A}nalysis, {R}andom {F}ields and
  {A}pplications {VI}}, volume~63 of {\em Progr. Probab.}, pages 23--35.
  Birkh\"{a}user/Springer Basel AG, Basel, 2011.

\bibitem{MR2765423}
Vladimir Bogachev, Giuseppe~Da Prato, and Michael R\"{o}ckner.
\newblock Uniqueness for solutions of {F}okker-{P}lanck equations on infinite
  dimensional spaces.
\newblock {\em Comm. Partial Differential Equations}, 36(6):925--939, 2011.

\bibitem{MR2111320}
Giuseppe Da~Prato.
\newblock {\em Kolmogorov equations for stochastic {PDE}s}.
\newblock Advanced Courses in Mathematics. CRM Barcelona. Birkh\"{a}user
  Verlag, Basel, 2004.

\bibitem{MR1985790}
Giuseppe Da~Prato and Jerzy Zabczyk.
\newblock {\em Second order partial differential equations in {H}ilbert
  spaces}, volume 293 of {\em London Mathematical Society Lecture Note Series}.
\newblock Cambridge University Press, Cambridge, 2002.

\bibitem{MR3236753}
Giuseppe Da~Prato and Jerzy Zabczyk.
\newblock {\em Stochastic equations in infinite dimensions}, volume 152 of {\em
  Encyclopedia of Mathematics and its Applications}.
\newblock Cambridge University Press, Cambridge, second edition, 2014.

\bibitem{MR1681462}
Gerald~B. Folland.
\newblock {\em Real analysis}.
\newblock Pure and Applied Mathematics (New York). John Wiley \& Sons, Inc.,
  New York, second edition, 1999.
\newblock Modern techniques and their applications, A Wiley-Interscience
  Publication.

\bibitem{MR2295424}
Avner Friedman.
\newblock {\em Stochastic differential equations and applications}.
\newblock Dover Publications, Inc., Mineola, NY, 2006.
\newblock Two volumes bound as one, Reprint of the 1975 and 1976 original
  published in two volumes.

\bibitem{MR2479730}
L.~Gawarecki, V.~Mandrekar, and B.~Rajeev.
\newblock Linear stochastic differential equations in the dual of a
  multi-{H}ilbertian space.
\newblock {\em Theory Stoch. Process.}, 14(2):28--34, 2008.

\bibitem{MR2590157}
L.~Gawarecki, V.~Mandrekar, and B.~Rajeev.
\newblock The monotonicity inequality for linear stochastic partial
  differential equations.
\newblock {\em Infin. Dimens. Anal. Quantum Probab. Relat. Top.},
  12(4):575--591, 2009.

\bibitem{MR2560625}
Leszek Gawarecki and Vidyadhar Mandrekar.
\newblock {\em Stochastic differential equations in infinite dimensions with
  applications to stochastic partial differential equations}.
\newblock Probability and its Applications (New York). Springer, Heidelberg,
  2011.

\bibitem{MR771478}
Kiyosi It\^{o}.
\newblock {\em Foundations of stochastic differential equations in
  infinite-dimensional spaces}, volume~47 of {\em CBMS-NSF Regional Conference
  Series in Applied Mathematics}.
\newblock Society for Industrial and Applied Mathematics (SIAM), Philadelphia,
  PA, 1984.

\bibitem{MR1465436}
Gopinath Kallianpur and Jie Xiong.
\newblock {\em Stochastic differential equations in infinite-dimensional
  spaces}, volume~26 of {\em Institute of Mathematical Statistics Lecture
  Notes---Monograph Series}.
\newblock Institute of Mathematical Statistics, Hayward, CA, 1995.
\newblock Expanded version of the lectures delivered as part of the 1993
  Barrett Lectures at the University of Tennessee, Knoxville, TN, March 25--27,
  1993, With a foreword by Balram S. Rajput and Jan Rosinski.

\bibitem{MR1121940}
Ioannis Karatzas and Steven~E. Shreve.
\newblock {\em Brownian motion and stochastic calculus}, volume 113 of {\em
  Graduate Texts in Mathematics}.
\newblock Springer-Verlag, New York, second edition, 1991.

\bibitem{MR570795}
N.~V. Krylov and B.~L. Rozovski\u{\i}.
\newblock Stochastic evolution equations.
\newblock In {\em Current problems in mathematics, {V}ol. 14 ({R}ussian)},
  pages 71--147, 256. Akad. Nauk SSSR, Vsesoyuz. Inst. Nauchn. i Tekhn.
  Informatsii, Moscow, 1979.

\bibitem{MR1472487}
Hiroshi Kunita.
\newblock {\em Stochastic flows and stochastic differential equations},
  volume~24 of {\em Cambridge Studies in Advanced Mathematics}.
\newblock Cambridge University Press, Cambridge, 1997.
\newblock Reprint of the 1990 original.

\bibitem{MR2020294}
Philip~E. Protter.
\newblock {\em Stochastic integration and differential equations}, volume~21 of
  {\em Applications of Mathematics (New York)}.
\newblock Springer-Verlag, Berlin, second edition, 2004.
\newblock Stochastic Modelling and Applied Probability.

\bibitem{MR1837298}
B.~Rajeev.
\newblock From {T}anaka's formula to {I}to's formula: distributions, tensor
  products and local times.
\newblock In {\em S\'{e}minaire de {P}robabilit\'{e}s, {XXXV}}, volume 1755 of
  {\em Lecture Notes in Math.}, pages 371--389. Springer, Berlin, 2001.

\bibitem{MR3063763}
B.~Rajeev.
\newblock Translation invariant diffusion in the space of tempered
  distributions.
\newblock {\em Indian J. Pure Appl. Math.}, 44(2):231--258, 2013.

\bibitem{brajeev-arxiv}
B.~Rajeev.
\newblock Translation invariant diffusions and stochastic partial differential
  equations in $\mathcal{S}^{\prime}$.
\newblock {\em arXiv:1901.00277v2 [math.PR]}, pages 1--56, 2019.

\bibitem{MR1999259}
B.~Rajeev and S.~Thangavelu.
\newblock Probabilistic representations of solutions to the heat equation.
\newblock {\em Proc. Indian Acad. Sci. Math. Sci.}, 113(3):321--332, 2003.

\bibitem{MR2373102}
B.~Rajeev and S.~Thangavelu.
\newblock Probabilistic representations of solutions of the forward equations.
\newblock {\em Potential Anal.}, 28(2):139--162, 2008.

\bibitem{MR1157815}
Walter Rudin.
\newblock {\em Functional analysis}.
\newblock International Series in Pure and Applied Mathematics. McGraw-Hill,
  Inc., New York, second edition, 1991.

\bibitem{MR2932517}
Daniel~W. Stroock.
\newblock {\em Partial differential equations for probabilists}, volume 112 of
  {\em Cambridge Studies in Advanced Mathematics}.
\newblock Cambridge University Press, Cambridge, 2012.
\newblock Paperback edition of the 2008 original.

\bibitem{MR2190038}
Daniel~W. Stroock and S.~R.~Srinivasa Varadhan.
\newblock {\em Multidimensional diffusion processes}.
\newblock Classics in Mathematics. Springer-Verlag, Berlin, 2006.
\newblock Reprint of the 1997 edition.

\bibitem{MR3483747}
Stefan Tappe.
\newblock Affine realizations with affine state processes for stochastic
  partial differential equations.
\newblock {\em Stochastic Process. Appl.}, 126(7):2062--2091, 2016.

\bibitem{MR2296978}
Fran\c{c}ois Tr\`eves.
\newblock {\em Topological vector spaces, distributions and kernels}.
\newblock Dover Publications, Inc., Mineola, NY, 2006.
\newblock Unabridged republication of the 1967 original.

\bibitem{MR876085}
John~B. Walsh.
\newblock An introduction to stochastic partial differential equations.
\newblock In {\em \'{E}cole d'\'{e}t\'{e} de probabilit\'{e}s de
  {S}aint-{F}lour, {XIV}---1984}, volume 1180 of {\em Lecture Notes in Math.},
  pages 265--439. Springer, Berlin, 1986.

\end{thebibliography}

\end{document}